\documentclass[letterpaper,11pt,oneside,reqno]{amsart}

\usepackage[english]{babel}

\usepackage{amsmath,amssymb,amsthm,amsfonts}
\usepackage{graphicx,color}
\usepackage[colorlinks=true,linkcolor=blue,citecolor=red]{hyperref}
\usepackage{bbm}
\usepackage{upgreek}
\usepackage[mathscr]{euscript}
\usepackage{textcomp}

\allowdisplaybreaks
\numberwithin{equation}{section}

\usepackage{tikz}
\usetikzlibrary{shapes,arrows,positioning,decorations.markings}

\usepackage{array}
\usepackage{adjustbox}
\usepackage{cleveref}
\usepackage{enumerate}
\usepackage{caption}
\usepackage[DIV=12]{typearea}

\usepackage[shadow=true]{todonotes}

\newtheorem{proposition}{Proposition}[section]
\newtheorem{lemma}[proposition]{Lemma}
\newtheorem{corollary}[proposition]{Corollary}
\newtheorem{theorem}[proposition]{Theorem}

\theoremstyle{definition}

\newtheorem{remark}[proposition]{Remark}
\newcommand{\bra}[1]{\left\langle #1\right|}
\newcommand{\ket}[1]{\left|#1\right\rangle}

\tikzset{bgplaq/.style={fill=lightgray!30!white}}
\tikzset{bggplaq/.style={fill=lightgray!70!white}}
\tikzset{bline/.style={line width=0.8mm}}
\tikzset{fline/.style={}}
\tikzset{gbline/.style={line width=0.8mm,draw=black!35!green}}
\tikzset{gline/.style={line width=0.3mm,draw=black!35!green}}
\tikzset{bbline/.style={line width=0.8mm,draw=black!35!blue}}
\tikzset{bline/.style={line width=0.3mm,draw=black!35!blue}}
\tikzset{wbline/.style={line width=0.6mm,draw=lightgray!10!white}}
\tikzset{wline/.style={draw=lightgray!10!white}}
\tikzset{part/.style={thick,solid,circle,fill=black,draw=black,inner sep=1.8pt,outer sep=3pt}}
\tikzset{hole/.style={thick,solid,circle,draw=black,fill=white,inner sep=1.8pt,outer sep=3pt}}
\tikzset{arrow/.style={postaction={decorate,thick,decoration={markings,mark = at position #1 with {\arrow{>}}}}},arrow/.default=0.5}

\tikzset{cross/.style={postaction={decorate,thick,decoration={markings,mark = at position #1 with{\draw (-2pt,-2pt) -- (2pt,2pt);\draw (2pt,-2pt) -- (-2pt,2pt);}}}}}
\tikzset{distort/.style={cm={1,0,-\slt,\sltb,(0,0)}}}
\def\slt{0.2}
\pgfmathsetmacro{\sltb}{sqrt(1-\slt*\slt)}

\newcommand\gbull[3]{\filldraw[fill=black!35!green, draw=black] (#1,#2) circle (#3cm);}

\def\part{\begin{tikzpicture}[scale=0.5,baseline=-1.7] \gbull{0}{0}{0.12}; \end{tikzpicture}}

\allowdisplaybreaks[1]

\newcommand{\pf}{\mathop{\rm Pf}}

\renewcommand{\leq}{\leqslant}
\renewcommand{\geq}{\geqslant}

\definecolor{dred}{RGB}{176,0,0}

\begin{document}

\title[Refined Littlewood identity for spin Hall--Littlewood functions]
{Refined Littlewood identity for spin Hall--Littlewood symmetric rational functions}
	
	\author{Svetlana Gavrilova}
	
\keywords{}
	
	\maketitle
	
	\thispagestyle{empty}
	
\begin{abstract}
	Fully inhomogeneous spin Hall--Littlewood symmetric rational
	functions $F_\lambda$ are multiparameter deformations of the classical Hall--Littlewood symmetric polynomials and can be viewed as partition functions in $\mathfrak{sl}(2)$ higher spin six vertex models.
	
	We obtain a refined Littlewood identity
	expressing a 
	weighted sum of $F_\lambda$'s over all signatures $\lambda$ with even multiplicities
	as a certain Pfaffian. This Pfaffian can be derived as
	a partition function of the six vertex model in a triangle
	with suitably decorated domain wall boundary conditions.
	The proof
	is based on the Yang--Baxter equation.
\end{abstract}

\maketitle

\noindent

\section{Introduction}
\label{sec:intro}
\subsection{Background}

In the present paper we deal with summation 
identities for spin Hall--Little\-wood symmetric
rational functions. These functions
arise
as 
partition functions
of square lattice integrable
vertex models
related to the quantum group $U_q(\widehat{\mathfrak{sl}_2})$. This description originally appeared in \cite{Bor17}, \cite{BP18}.

The spin Hall--Littlewood functions
also can be identified with Bethe Ansatz eigenfunctions 
of the higher spin six vertex model on $\mathbb{Z}$, cf. \cite[Ch. VII]{KBI93}.
They also appear as eigenfunctions of certain stochastic particle systems
\cite{Pov13}, \cite{BCPS15}, \cite{CP16}. Following
\cite{Bor17}, \cite{BP18} and subsequent works, we treat spin Hall--Littlewood functions and their relatives from the point of
view of the theory of symmetric functions.
A classical reference on the theory of symmetric functions is the book
\cite{Mac95} where Schur, Hall--Littlewood, and Macdonald symmetric
polynomials
and symmetric functions are developed and various identities for them are formulated or proved.

One of the common features 
for most
families of symmetric polynomials
is a \emph{Littlewood type summation identity}. For example, the Schur symmetric polynomials $s_{\lambda}$ satisfy the following Littlewood identity:
\begin{equation*}
\sum_{\lambda'\ \text{even}} s_{\lambda}(u_1, \dots, u_m)
=
\prod_{1 \leq i<j\leq m}
(1 - u_i u_j)^{-1}.
\end{equation*}
Here the summation is over all signatures $\lambda=(\lambda_1\ge \ldots\ge \lambda_m\ge0 )$ such that all parts of its conjugate $\lambda'$ are even, or equivalently, all part-multiplicities of $\lambda$ are even.
For a comprehensive study of Littlewood identities for
Hall–Littlewood polynomials, we refer the reader to \cite[Ch. III]{Mac95}, \cite{War06} and to \cite{RW21} for recent developments concerning boxed Littlewood formulae for
Macdonald polynomials.

Moreover, Littlewood identities are important for integrable probability: they appear as a key tool for studying half-space integrable models related to the corresponding half-space Macdonald processes, see \cite{BBC20}, \cite{BBCW18}, \cite{BZ19}.

We study \emph{refinements} of Littlewood type identities,
which are derived by inserting an extra factor
into each term of the summation in the left-hand side.
The expression for the right-hand side, in turn, also gets more complicated: it becomes a Pfaffian. Earlier, a number of Pfaffian formulas for partition functions of the six vertex model were obtained by Kuperberg in \cite{Kup02}. We follow a method for proving refined (Cauchy and Littlewood type) identities introduced in \cite{WZJ16}, which is based on the Yang--Baxter equation.

One of the applications of refined Cauchy identities for Macdonald polynomials is a possibility to compute the expectations of observables for Macdonald
measures. Namely, they can be expressed as a certain determinantal formula independent of the parameter $q$.
This result goes back to 
\cite{KN99}, see also \cite{War08}, \cite{Bor18}, \cite{Pet21}.
It would be nice to see if this $q$-independence extends to the Pfaffian case. Moreover, it would be interesting to employ our result for the analysis of half-space models of integrable probability as in \cite{BBC20}, \cite{BBCW18}, \cite{BZ19}, but this application is outside of the scope of the present work.

\begin{remark}
	After completing this manuscript, Littlewood type identities for stable spin Hall--Littlewood polynomials, which are specializations of our functions, were also applied in \cite{CD21} for introducing the half-space Yang-Baxter random field and studying related dynamic systems.
\end{remark}

\subsection{Refined Littlewood identity for spin Hall--Littlewood functions}

One of possible ways to define the fully inhomogeneous 
spin Hall--Littlewood 
symmetric rational functions is the following
symmetrization form introduced in \cite{BP18}:
\begin{equation*}
\begin{split}
F_\lambda(u_1,\ldots,u_N )
&=
\sum_{\sigma\in S_N}
\sigma
\Biggl( 
\prod_{1\le i<j\le N}\frac{u_i-t u_j}{u_i-u_j}\,
\prod_{i=1}^{N}
\biggl(
\frac{1-t}{1-s_{\lambda_i}u_i}
\prod_{j=0}^{\lambda_i-1}\frac{ u_i-s_j}{1- s_j u_i}
\biggr)
\Biggr),
\end{split}
\end{equation*}
where 
$\lambda=(\lambda_1\ge\ldots\ge \lambda_N\ge0 )$ is a signature, that is, a sequence of weakly decreasing nonnegative integers.
Here $\sigma\in S_N$
acts by permuting the variables $u_i$'s.
The function $F_\lambda$ depends on the ``quantum parameter''
$t \in (0, 1)$, the variables $u_j$ and the inhomogeneities $s_x$, where $x\in \mathbb{Z}_{\ge0}$.
By setting $s_x=0$ for all $x$, we obtain the reduction to the case of usual
Hall--Littlewood symmetric polynomials.

Our main result is a generalization of the refined Littlewood identity \eqref{class_id} to the case of the spin Hall--Littlewood functions. 

To formulate the result we need some notation given below. Namely, $m_0(\lambda)$ is the number of parts in signature $\lambda$ equal to zero, and 
$(a;t)_k=(1-a)(1-at)\ldots(1-at^{k-1})$ is the $t$-Pochhammer symbol.
\begin{theorem} \label{lit_id} 
	Let $\gamma\ne0$ be an arbitrary complex number and let variables $u_1, \dots, u_{2n}$ satisfy restrictions \eqref{eq:admissible_u} below which are needed for some convergence conditions.
	Then spin Hall--Littlewood symmetric rational functions satisfy the following refined Littlewood identity:
	\begin{equation} \label{littlewood_id}
	\begin{split}
	&
	\sum_{\lambda: m_{i}(\lambda) \in 2 \mathbb{Z}_{\geq 0}} \frac{1}{( t; t)_{m_{0}(\lambda)}} \prod_{j = 1}^{m_{0}(\lambda) / 2} (1 - s_{0}^2 \gamma^{-1} t^{2j - 2})(1 - \gamma t^{2j-1})  \prod_{j=1}^{2n} (1 - s_0 u_j)		\\&\hspace{70pt}\times
	\prod_{i = 1}^{\infty} \prod_{j = 1}^{m_{i}(\lambda) / 2} \frac{1 - s_{i}^2 t^{2j - 2}}{1 - t^{2j}} F_{\lambda} (u_1, \dots, u_{2n})
	\; \;\;
	= 
	\prod_{1 \leq i<j \leq 2n}
	\left(
	\frac{1-t u_i u_j}{u_i - u_j}
	\right)
	\\&\hspace{40pt}\times
	\pf_{1\leq i < j \leq 2n}
	\left[
	\frac{ (u_i - u_j)((1 - t)(1 - s_0 u_i)(1- s_0 u_j) + (1 - \gamma)(t-s_{0}^{2}\gamma^{-1})(1-u_i u_j))}
	{(1-u_i u_j) (1-t u_i u_j)}
	\right].
	\end{split}
	\end{equation}
\end{theorem}

\subsection{Sketch of proof}
Our approach follows the work of M. Wheeler and P. Zinn-Justin \cite{WZJ16} and the work of L. Petrov \cite{Pet21}. Namely, we represent spin Hall--Littlewood rational functions as certain partition functions, using the integrable model of deformed bosons. 
Then we consider a partition function that can be identified with some weighted sum of spin Hall--Littlewood functions. After that we use the Yang--Baxter equation to replace our partition function with equal partition function of the six vertex model with finitely many vertices. It allows us to prove some properties of the function and to present a particular function (in our case it is some certain Pfaffian) with the same properties.
Finally, we use Lagrange interpolation to verify that our properties determine the function uniquely. This technique goes back to Izergin and Korepin \cite{Ize87}, \cite{KBI93}.

\subsection{Notation}
\label{sec:not}

Let us introduce some notation.

Each signature $\lambda=(\lambda_1\ge \lambda_2\ge \ldots\ge \lambda_N\ge0 )$ can be written in multiplicative form as $\lambda=(0^{m_0} 1^{m_1} 2^{m_2} \dots)$, where $m_i$ is the multiplicity of $i$ in $\lambda$. Throughout the paper we will use this notation correspondence.

We often deal with tensor products of the same space, so we use upper indices to point out in which component a certain operator acts. For example, if $\omega$ is a  $4\times4$ matrix and we have a $2^n$-dimensional tensor power of $n$ $2$-dimensional spaces, then $\omega^{(i, {i+1})}= \mathbbm{1}^{\otimes ({i - 1})} \otimes \omega \otimes \mathbbm{1}^{\otimes ({n - i - 1})}$ where $\mathbbm{1}$  is a $2\times2$ identity matrix.

\subsection{Organization of the paper} In \Cref{sec:vertex_and_sHL} we recall the basic notation,
definitions and properties of the spin Hall--Littlewood rational symmetric functions and the integrable model related to them.
In \Cref{sec:sHL_proof} we prove the refined Littlewood identity 
for the spin Hall--Littlewood functions.
Finally, in \Cref{sec: spec}
we discuss the reduction of the result to the classical family of Hall-Littlewood symmetric functions and write the non-refined case of our identity.

\subsection{Acknowledgements.}
I would like to thank Leonid Petrov for setting the problem, valuable discussions and constant attention to this work.
The author is partially supported by International Laboratory of Cluster Geometry NRU HSE, RF Government grant, ag. \textnumero 075-15-2021-608.

\section{Higher spin six vertex model weights and spin Hall--Littlewood functions} \label{sec:vertex_and_sHL}

In this section we introduce certain model with higher spin six vertex weights and explain how to build spin Hall--Littlewood functions in terms of this model.

\subsection{Definition of the model}
Consider an infinite dimensional vector space $V$:
\begin{align*}
\label{space}
V 
=
{\rm Span} 
\left\{
\ket{m_0}_0
\otimes
\ket{m_1}_1 
\otimes
\ket{m_2}_2 
\otimes
\cdots
\right\},
\qquad
m_i \in \mathbb{Z}\ \forall\ i \geq 1,
\end{align*}
where only finitely many of the $m_i$ are nonzero.
It is convenient to think that 
$m_i$ represents the number of particles at site $i$.
Note that
if $m_i$ are obtained as multiplicities of some signature $\lambda$, then obviously we have $m_i \in \mathbb{Z}_{\geq 0}$  and in this case we denote the corresponding state by $|\lambda\rangle \in V$ or $\langle \lambda | \in V^*$. However, for our purposes it makes sense to work with negative integers, too.

Also, we set up a two-dimensional auxiliary vector space $W = \mathbb{C}^2$ and its basis denoted by
$
|0\rangle $ and $ |1\rangle
$.
Then the higher spin six vertex model weights $w_{u,s_i}(i_1,j_1;i_2,j_2)$ can be implemented by considering the operators $L_{u, s_i}$ acting in $W \otimes V_i$ where by $V_i$ we denote the $i^{th}$ factor of $V$. Namely,
the weight $w_{u,s}(i_1,j_1;i_2,j_2)$ is defined as $ \langle j_2| \otimes \langle i_2| \; L_{u, s} \; |j_1 \rangle \otimes |i_1\rangle $, where $i_1,i_2\in \mathbb{Z}_{\ge0}$
and $j_1,j_2\in \left\{ 0,1 \right\}$.
Graphically, we can represent this operator as in \Cref{fig:w_weights}.
\begin{figure}[htpb]
	\centering
	\includegraphics{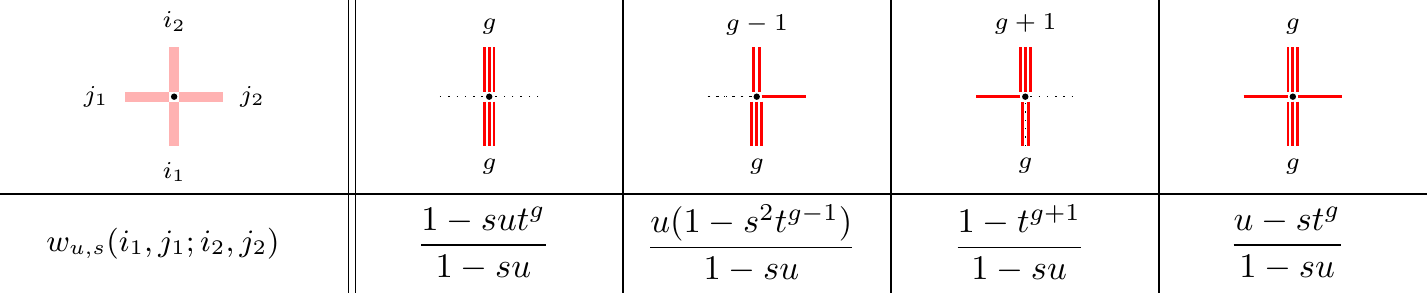}
	\caption{Vertex weights $w_{u,s}(i_1,j_1;i_2,j_2)$.  Here $i_1,i_2\in \mathbb{Z}_{\ge0}$
		and $j_1,j_2\in \left\{ 0,1 \right\}$.}
	\label{fig:w_weights}
\end{figure}

Likewise, let us define operators $L^*_{v, s}$ and the corresponding weights as follows:
\begin{equation}
\label{eq:w_wstar_relation}
w_{v,s}^*(i_1,j_1;i_2,j_2)= \langle j_2| \otimes \langle i_2| \; L^*_{v, s} \; |j_1 \rangle \otimes |i_1\rangle=\frac{(s^2;t)_{i_1}(t;t)_{i_2}}{(s^2;t)_{i_2}(t;t)_{i_1}}\,w_{v,s}(i_2,j_1;i_1,j_2).
\end{equation}

\begin{figure}[htpb]
	\centering
	\includegraphics{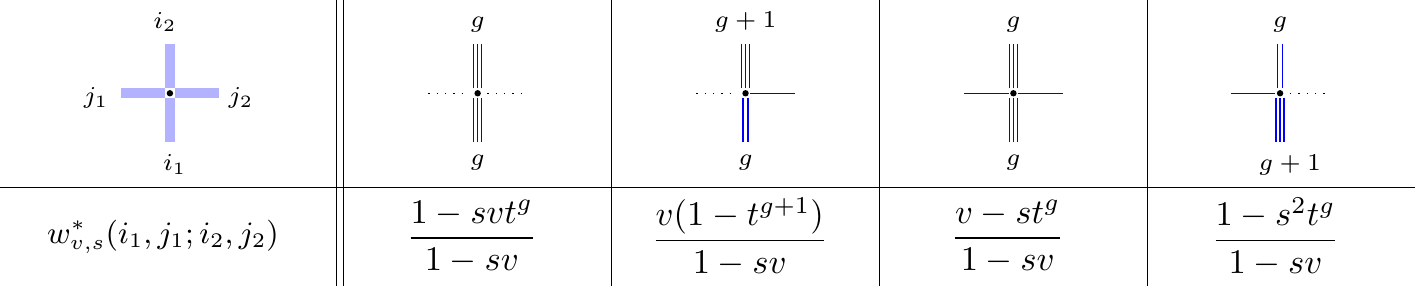}
	\caption{Vertex weights $w_{v,s}^*(i_1,j_1;i_2,j_2)$. 	Here $i_1,i_2\in \mathbb{Z}_{\ge0}$
		and $j_1,j_2\in \left\{ 0,1 \right\}$.}
	\label{fig:w_star_weights}
\end{figure}

Impose the following restrictions on the variables $u_1, \dots, u_N$:
\begin{equation}
\label{eq:admissible_u}
\left|
\frac{ u_i-s_x}{1-s_x u_i}
\right|\le 1-\varepsilon<1
\qquad \textnormal{for all $i$ and all $x=0,1,2,\ldots $.}
\end{equation}

Define the following transfer matrices acting on $W \otimes V$:
\begin{align*}
T(u)
&=
\prod_{i=0}^{\infty}
L_{u, s_i} =
\begin{pmatrix}
0 & T_{+}(x)
\medskip
\\
0 & T_{-}(x)
\end{pmatrix}
\in
{\rm End}(W \otimes V_0 \otimes V_1 \otimes \cdots),
\\
T^{*}(u)
&=
\prod_{i=0}^{\infty}
L^*_{u, s_i} =
\begin{pmatrix}
T^{*}_{+}(x) & 0
\medskip
\\
T^{*}_{-}(x) & 0
\end{pmatrix}
\in
{\rm End}(W \otimes V_0 \otimes V_1 \otimes \cdots).
\end{align*}
See \Cref{fig:T} for an illustration.
\begin{figure}[htpb]
	\centering
	\includegraphics{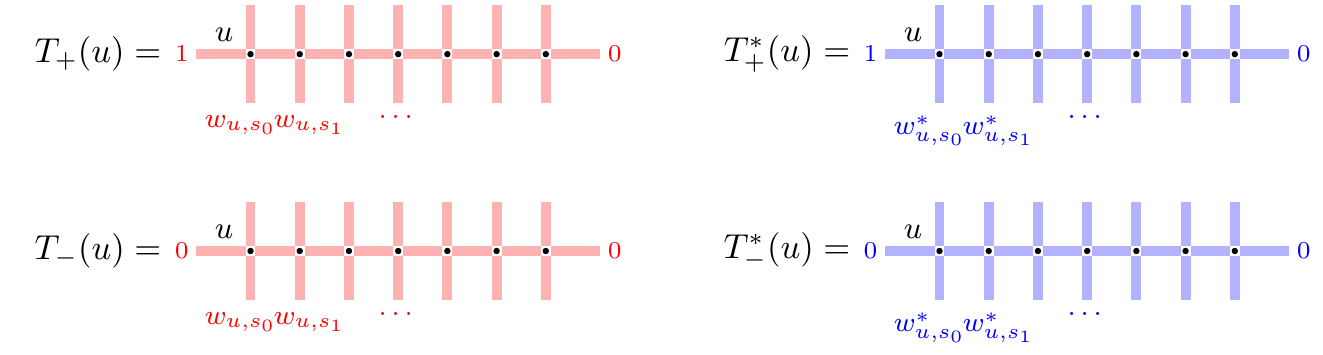}
	\caption{Graphical representation of $T$ and $T^{*}$ operators.}
	\label{fig:T}
\end{figure}
\begin{remark}
	Since we require the convergence condition \eqref{eq:admissible_u}, it follows that operators $T_{+}$ and $T^{*}_{+}$ have the vanishing property. Namely, any path that goes endlessly to the right produces the weight equal to zero, so we can forbid such paths. This means that we do not actually need to write $0$ on the right boundary in \Cref{fig:T}.
\end{remark}

Let us introduce the $R$-matrix of the six vertex model:
\begin{align*}
\label{Rmat}
R_{z}
=
\begin{pmatrix}
1 & 0 & 0 &  \frac{(1-t)z}{1-z} \\
0 & \frac{1-t z}{1-z} & 0 & 0 \\
0 & 0 & \frac{1-t z}{1-z} & 0 \\
\frac{1-t}{1-z} & 0 & 0 & t
\end{pmatrix}
\in 
{\rm End}(W \otimes W).
\end{align*}

Graphically, we denote the action of this operator by cross vertices with the weights given below:
\begin{figure}[htpb]
	\centering
	\includegraphics{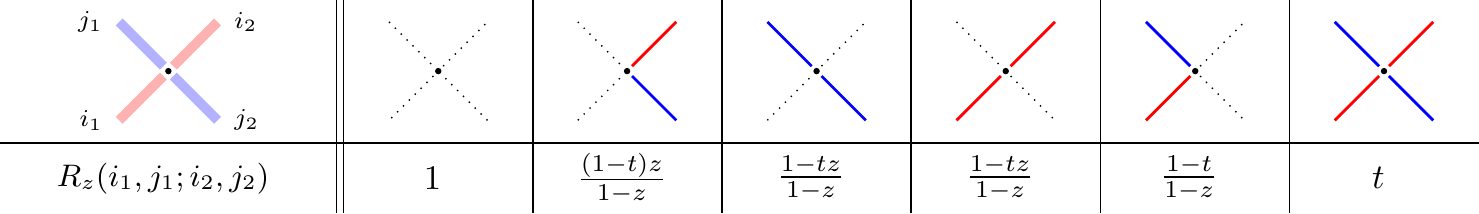}
	\caption{Cross vertex weights $R_z$. Here we have $i_1,j_1,i_2,j_2\in \left\{ 0,1 \right\}$.}
	\label{fig:R_weights}
\end{figure}

\begin{proposition}[Yang--Baxter equation]
	\label{prop:YBE_w_wstar}
	For any $i_1,i_2,j_1,j_2\in  \left\{ 0,1 \right\}$ and 
	$i_3,j_3\in \mathbb{Z}_{\ge0}$ we have
	\begin{equation}
	\label{eq:YBE_w_wstar}
	\begin{split}
	&\sum_{k_1,k_2,k_3}R
	_{u v}(i_2, i_1; k_2, k_1)\,
	w^*_{v,s} (i_3, k_1; k_3, j_1)\,
	w_{u,s}(k_3,k_2; j_3,j_2) \\
	&\hspace{100pt}
	= 
	\sum_{k'_1,k'_2,k'_3} w^*_{v,s} (k'_3, i_1; j_3, k'_1)\,
	w_{u,s}(i_3,i_2; k'_3,k'_2)\,
	R_{u v}(k'_2, k'_1; j_2, j_1),
	\end{split}
	\end{equation}
	or, equivalently,
	\begin{equation}
	\label{RLL}
	R_{uv}^{(12)} L^{*(1)}_{v, s} L^{(2)}_{u, s} = L^{*(1)}_{v, s} L^{(2)}_{u, s} R_{uv}^{(12)}.
	\end{equation}
\end{proposition}

\begin{proof}
	The proof is by direct computations, and we omit them.
\end{proof}

\begin{figure}[htpb]
	\centering
	\includegraphics{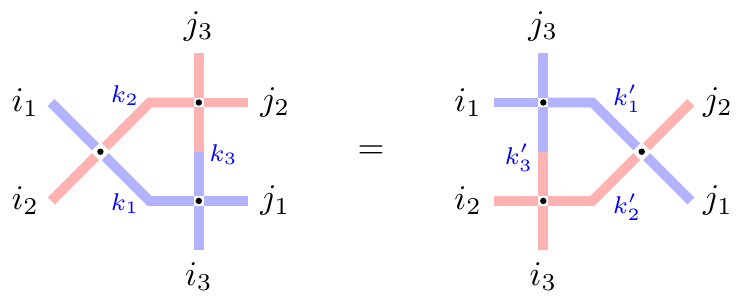}
	\caption{Graphical illustration of the Yang--Baxter equation \eqref{eq:YBE_w_wstar}.}
	\label{fig:YBE_w_wstar}
\end{figure}

\begin{remark}
	It is important that cross vertex weights in the Yang--Baxter equation \eqref{eq:YBE_w_wstar} do not depend on $s$. This observation allows us to iterate this interchange relation horizontally and get an
	equation for $T$ and $T^*$ operators which is the same as \eqref{RLL} with $L_{u, s}$ and $L^{*}_{v, s}$ replaced by $T(u)$ and $T^{*}(v)$, respectively.
\end{remark}

\subsection{Spin Hall-Littlewood functions}
Now we are able to give a definition of the spin Hall--Littlewood rational functions in terms of our model. Namely, they are given by the following formula:
\begin{align*}
F_{\lambda}(u_1,\dots,u_N)
&=
\bra{0}
T_{+}(u_1)
\dots
T_{+}(u_N)
\ket{\lambda}.
\end{align*}

In other words, we consider the weighted sum over all the up-right paths ensembles in $\mathbb{Z}_{\ge0}\times \left\{ 1,\ldots,N  \right\}$ with the following properties:

\begin{enumerate}[\bf1.\/]
	
	\item Each path comes from the left edge and reaches the top boundary at the corresponding coordinate $\lambda_i$.
	
	\item No two paths can share the same horizontal line.
	
	\item In the vertex $(x, i) \in \mathbb{Z}_{\ge0}\times \left\{ 1,\ldots,N  \right\}$ we take the weight $w_{u_i, s_x}$.
	
\end{enumerate}
The example of such an ensemble is given in \Cref{fig:F_ex}.
\begin{figure}[htpb]
	\centering
	\includegraphics{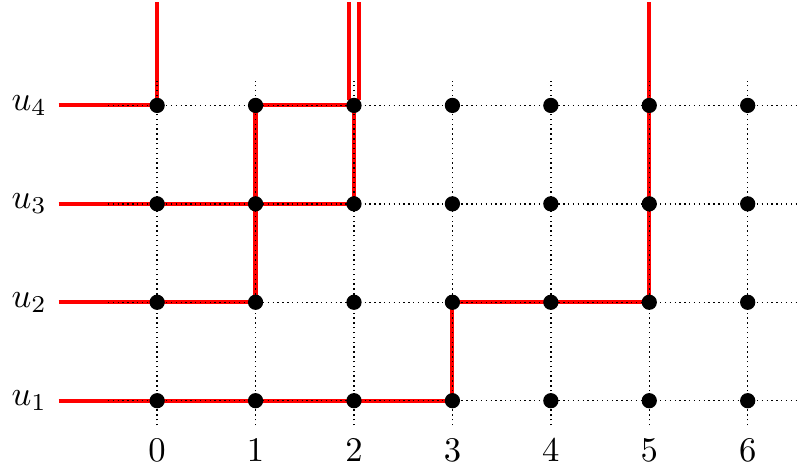}
	\caption{An example of a path configuration
		contributing to the partition function $F_\lambda(u_1,u_2,u_3,u_4)$, where
		$\lambda=(5,2,2,0)$.}
	\label{fig:F_ex}
\end{figure}
\subsection{Refinement}
One can define a generalization of this partition function by adding an extra parameter $\alpha \in \mathbb{C}$. Namely, consider a vector space \begin{align*}
\label{alpha-space}
V(\alpha) 
=
{\rm Span} 
\left\{
\ket{m_0+ \alpha}_0
\otimes
\ket{m_1}_1 
\otimes
\ket{m_2}_2 
\otimes
\cdots
\right\},
\qquad
m_i \in \mathbb{Z}\ \forall\ i \geq 1,
\end{align*}
and the corresponding family of partition functions $$F^{\alpha}_\lambda(u_1, \dots, u_N) = \bra{0; \alpha}
T_{+}(u_1)
\dots
T_{+}(u_N)
\ket{\lambda; \alpha},$$
where $\lambda$ is a signature and
\begin{align*}
|\lambda; \alpha\rangle
=
\ket{m_0(\lambda) + \alpha}_0
\otimes
\ket{m_1(\lambda)}_1
\otimes
\ket{m_2(\lambda)}_2
\otimes
\cdots
.
\end{align*}
Let $\gamma = t^{\alpha}$. Since the only difference between $F^{\alpha}_\lambda$ and $F_\lambda$ comes from different zero column weights, it is easy to express one partition function through another: 
\begin{equation}
\label{F_alpha}
\begin{split}
&F^{\alpha}_\lambda(u_1, \dots, u_N) = \frac{(\gamma t; t)_{m_0}}{( t, t)_{m_0}} \prod_{j=1}^{N} \frac{1 - \gamma s_0 u_j}{1 - s_0 u_j}
\left[ F_{\lambda}(u_1, \dots, u_N)
\Big\vert_{\text{$s_0\to \gamma s_0$}}\right],\\
&F_{\lambda}(u_1, \dots, u_N) = F^{0}_{\lambda}(u_1, \dots, u_N).
\end{split}
\end{equation}

\section{Refined Littlewood identity. Proof.}
\label{sec:sHL_proof}

In this section we prove the refined Littlewood identity (Theorem~\ref{lit_id}).
\subsection{A property of the transfer matrices}
\begin{lemma} \label{lem_T_op}
	Consider the following formal weighted sum of all states with even multiplicities:
	\begin{align*}
	\label{e-vect}
	\ket{\rm e;\alpha}
	=
	\sum_{\substack{
			m_{i}(\lambda) \in 2 \mathbb{Z}_{\geq 0} \\ m_0(\lambda) \in 2 \mathbb{Z} 
	}}
	\
	c_{\lambda}(t; \alpha)
	\ket{\lambda;\alpha},
	\end{align*}
	where the weights are given by
	\begin{align*}
	c_{\lambda}(\alpha,t) 
	= 
	\prod_{i = 1}^{\infty} \prod_{j = 1}^{m_{i}(\lambda) / 2} \frac{1 - s_{i}^2 t^{2j - 2}}{1 - t^{2j}}
	\times
	\left\{
	\begin{array}{cl}
	\displaystyle{
		\prod_{j = 1}^{m_{0}(\lambda) / 2} \frac{1 - s_{0}^2 \gamma t^{2j - 2}}{1 - \gamma t^{2j}}
	},
	&
	m_0(\lambda) \geq 0,
	\\ \\ 
	\displaystyle{
		\prod_{j=1}^{-m_0(\lambda)/2}
		\frac{1 - \gamma t^{-2j+2}}{1 - s_{0}^2 \gamma t^{-2j}}
	},
	&
	m_0(\lambda) \leq 0.
	\end{array}
	\right.
	\end{align*}
	
	Then the transfer matrices $T_{\pm}$ and $T^{*}_{\pm}$ have the following property:
	\begin{equation} \label{eq: lem_T_op}
	T_{+}\ket{\rm e;\alpha} = T^{*}_{+}\ket{\rm e;\alpha}, \quad\quad T_{-}\ket{\rm e;\alpha} = T^{*}_{-}\ket{\rm e;\alpha}.
	\end{equation}
\end{lemma}

\begin{proof}
	Take any signature $\mu$ and the corresponding state
	\begin{align*}
	\langle\mu; \alpha|
	=
	\bra{m_0(\mu) + \alpha}_0
	\otimes
	\bra{m_1(\mu)}_1
	\otimes
	\bra{m_2(\mu)}_2
	\otimes
	\cdots
	\end{align*}
	with $m_0(\mu) \in \mathbb{Z}$ and $m_i(\mu) \in \mathbb{Z}_{\geq 0}$ for all $i \geq 1$.
	Note that there exists a unique $\mu_{+}$ with even multiplicities such that $\bra{\rm \mu;\alpha}T_{+} \ket{\rm \mu_{+};\alpha} \neq 0$ or $\bra{\rm \mu;\alpha}T_{-} \ket{\rm \mu_{+};\alpha} \neq 0$ (which of these is nonzero depends on the parity of the sum over all the multiplicities). Also, denote by $\mu_{-}$ the unique signature with even multiplicities such that $\bra{\rm \mu;\alpha}T^{*}_{+} \ket{\rm \mu_{-};\alpha} \neq 0$ or $\bra{\rm \mu;\alpha}T^{*}_{-} \ket{\rm \mu_{-};\alpha} \neq 0$.
	For example, if $\mu = (6, 4, 4, 3, 2, 2, 0)$, then $\mu_{+} = (6, 6, 4, 4, 2, 2, 0, 0)$ and $\mu_{-} = (4, 4, 3, 3, 2, 2)$ (see \Cref{fig:ex} for an illustration).
	\begin{figure}[htpb]
		\centering
		\includegraphics{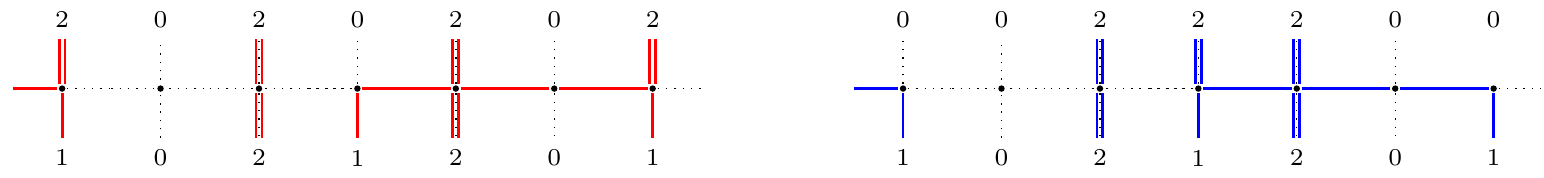}
		\caption{An illustration for the definition of $\mu_{+}$ and $\mu_{-}$ when $\mu = (6, 4, 4, 3, 2, 2, 0)$.}
		\label{fig:ex}
	\end{figure}
	
	So, we obtain 
	\begin{align*}
	\bra{\rm \mu;\alpha}T_{\pm} \ket{\rm e;\alpha} = c_{\mu_{+}} \bra{\rm \mu;\alpha}T_{\pm} \ket{\rm \mu_{+};\alpha} , \qquad \bra{\rm \mu;\alpha}T^{*}_{\pm} \ket{\rm e;\alpha} = c_{\mu_{-}} \bra{\rm \mu;\alpha}T^{*}_{\pm} \ket{\rm \mu_{-};\alpha}.
	\end{align*}
	It remains to check that \begin{equation}\label{T_pm_rel}
	c_{\mu_{+}} \bra{\rm \mu;\alpha}T_{\pm} \ket{\rm \mu_{+};\alpha} = c_{\mu_{-}} \bra{\rm \mu;\alpha}T^{*}_{\pm} \ket{\rm \mu_{-};\alpha}.
	\end{equation}
	This equality can be seen from the special case of equation~\eqref{eq:w_wstar_relation} and its analogue for the zero column.
	\begin{figure}[htpb]
		\centering
		\includegraphics{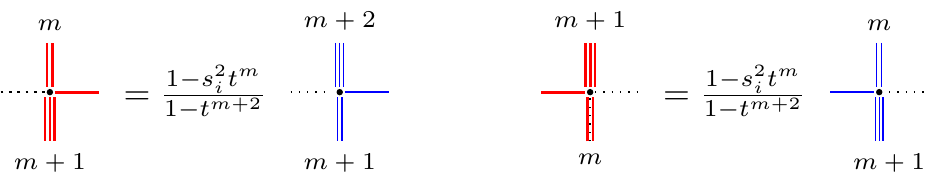}
		\includegraphics{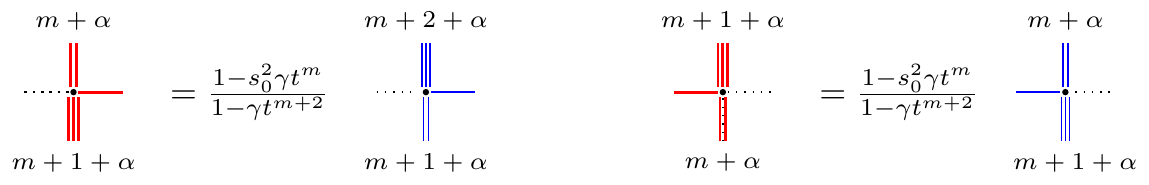}
		\caption{The correspondence between vertex weights in both sides of \eqref{T_pm_rel}.}
		\label{fig:pr_w}
	\end{figure}
	Indeed, to get $\bra{\rm \mu;\alpha}T_{\pm} \ket{\rm \mu_{+};\alpha}$ with given $\bra{\rm \mu;\alpha}T^{*}_{\pm} \ket{\rm \mu_{-};\alpha}$ we need to do the replacements as in \Cref{fig:pr_w}. These replacements produce some factor, meanwhile the ratio $c_{\mu_{+}}/c_{\mu_{-}}$ precisely compensates this factor. This concludes the proof.
\end{proof}

Graphically, our statement can be represented as in \Cref{fig:T_e}.
\begin{figure}[htpb]
	\centering
	\includegraphics{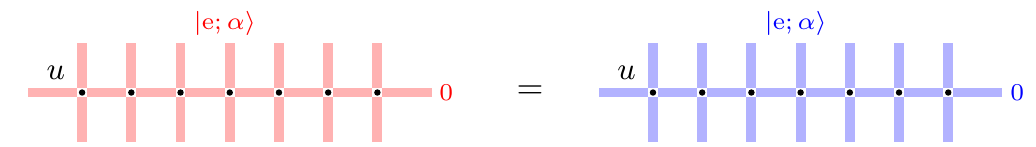}
	\caption{Graphical representation of equation \eqref{eq: lem_T_op}. Here we do not need to specify the state on the left boundary.}
	\label{fig:T_e}
\end{figure}

\subsection{Setting and transformation of the partition function}

Consider the following partition function which has an additional parameter $\alpha$:
\begin{figure}[htpb]
	\centering
	\includegraphics{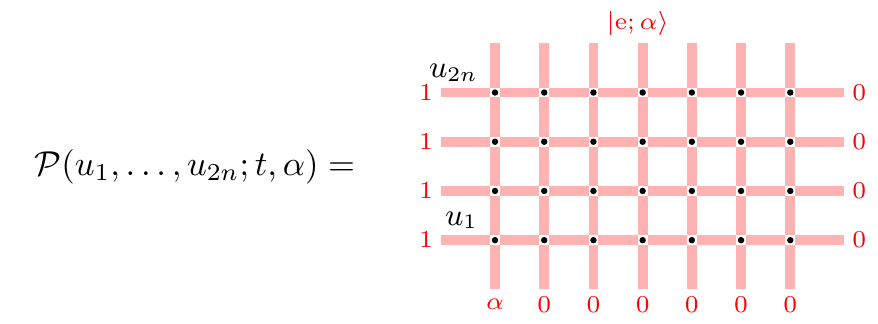}
	\caption{The left-hand side of the Littlewood identity expressed graphically as a partition function.}
	\label{fig:part_P}
\end{figure}

From the very definition we have
\begin{equation}
\label{part_alpha}
\mathcal{P}(u_1, \dots, u_{2n}; t, \alpha) 
=\sum_{\lambda: m_{i}(\lambda) \in 2 \mathbb{Z}_{\geq 0}} \prod_{j = 1}^{m_{0}(\lambda) / 2} \frac{1 - s_{0}^2 \gamma t^{2j - 2}}{1 - \gamma t^{2j}} \prod_{i = 1}^{\infty} \prod_{j = 1}^{m_{i}(\lambda) / 2} \frac{1 - s_{i}^2 t^{2j - 2}}{1 - t^{2j}} F^{\alpha}_{\lambda} (u_1, \dots, u_{2n}).
\end{equation}

Since we have up-right path ensembles and the left edge is occupied, it follows that only states with non-negative $m_0$ in $\ket{\rm e;\alpha}$ contribute to our summation.

Let us apply Lemma~\ref{lem_T_op} to the upper row. We get the first equality in \Cref{fig:12steps}.
\begin{figure}[htpb]
	\centering
	\includegraphics{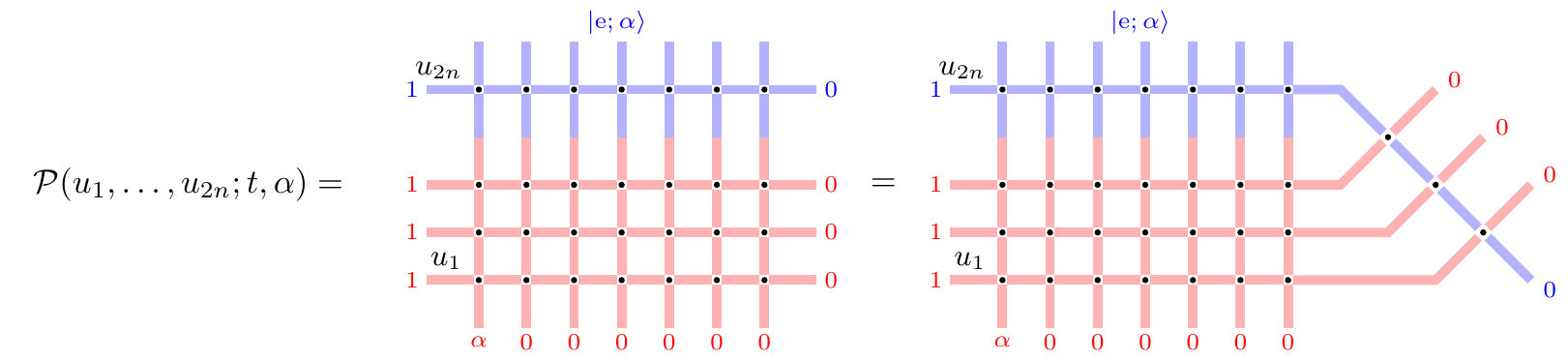}
	\caption{First step of the transformation of the partition function.}
	\label{fig:12steps}
\end{figure}
The second equality in \Cref{fig:12steps} holds due to the completely frozen cross part on the right, which has weight $1$.

Next, using the Yang--Baxter equation, one can move cross part of the partition to the left edge. Then we repeat this trick several times (see \Cref{fig:34steps} for an illustration).
\begin{figure}[htpb]
	\centering
	\includegraphics{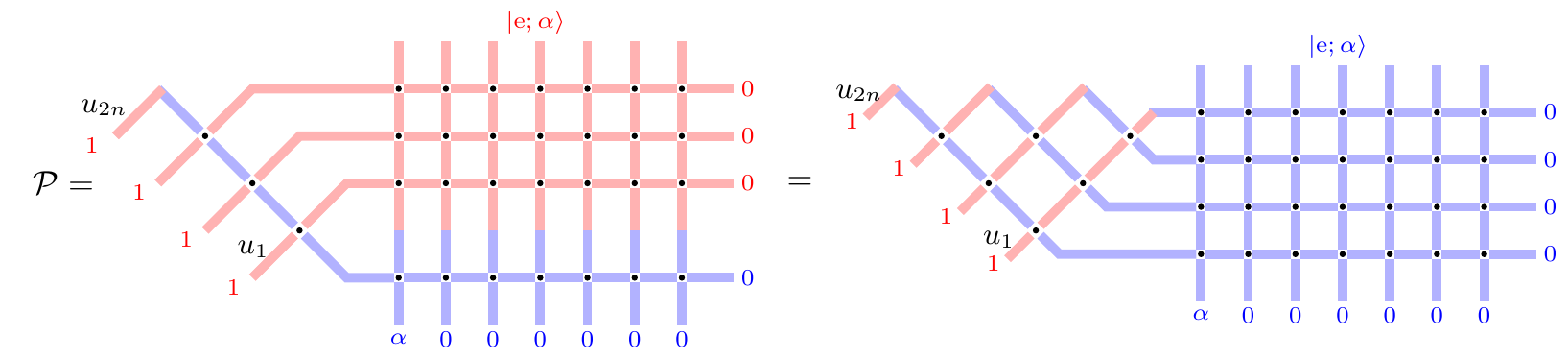}
	\caption{Next steps of the transformation of the partition function.}
	\label{fig:34steps}
\end{figure}

After moving all the crosses to the left, the partition function factorizes, and the blue frozen part on the right has weight 1. Thus, we obtain the partition function as in \Cref{fig:5step}. 
\begin{figure}[htpb]
	\centering
	\includegraphics{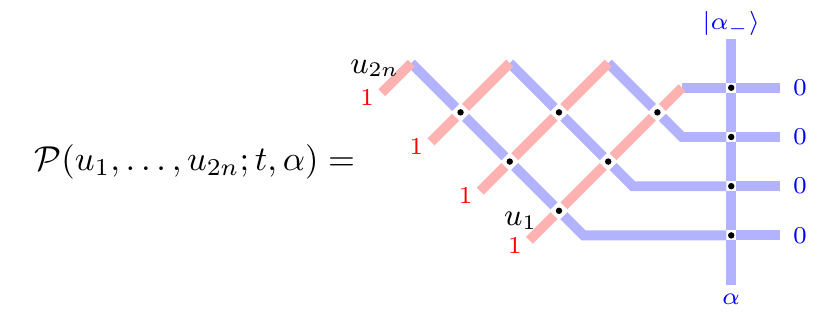}	\caption{Expression for the left-hand side of the Littlewood identity as a partition function of the inhomogeneous six vertex model with 
		weights $R_{u_iu_j}$ and decorated boundary conditions.}
	\label{fig:5step}
\end{figure}
The boundary vector $\ket{\rm \alpha_{-}}$ involved there can be expressed explicitly as the following weighted sum:
\begin{equation*}
\ket{\rm \alpha_{-}} = \sum_{k=0}^{\infty} \prod_{j=1}^{k} \frac{1 - \gamma t^{-2j+2}}{1 - s_{0}^2 \gamma t^{-2j}} \ket{ \alpha - 2k}.
\end{equation*}

\subsection{Properties of the partition function}
\begin{lemma}
	\label{prop_lem}
	Let $Z_{2n}$ denote the partition function $\mathcal{P} (u_1, \dots, u_{2n}; t, \alpha)$ as in \Cref{fig:5step} multiplied by $\prod_{1 \leq i < j \leq 2n} (1 - u_{i} u_{j}) \prod_{i = 1}^{2n} (1 - s u_i)$.
	Then $Z_{2n}$ possesses the following properties:
	\begin{enumerate}[\bf1.\/]
		\item $Z_{2n}$ is symmetric in $\{u_1, \dots, u_{2n}\}$.
		
		\item $Z_{2n}$ is a polynomial in $u_{2n}$ of degree $2n - 1$.
		
		\item Setting $u_{2n} = u_{2n - 1}^{-1}$, we have the recursion relation
		$$
		Z_{2n}|_{u_{2n} = u_{2n - 1}^{-1}} = (1 - t) (1 - \gamma s_{0} u_{2n}) (1 - \gamma s_{0} u_{2n-1}) \prod_{j = 1}^{2n-2} (1 - t u_{j} u_{2n}) (1 - t u_{j} u_{2n-1}) Z_{2n - 2}.
		$$
		
		\item Under the specialization $u_{2j - 1} = t$, $u_{2j} = 1 / t^2$ for $1 \leq j \leq n$, we have 
		$$Z_{2n}(t,  1 / t^2, \dots, t,   1 / t^2) = \gamma^{n} (t - 1)^{n^{2}}t^{-2n} (-(t - 1/t)^2 )^{n(n-1)/2}   (1 - s_{0} t^{-2})^n (1 -s_{0} t)^{n}.
		$$
		
		\item For $n=1$ we have
		$$Z_{2}(u_1, u_2) = (1 - t)(1 - \gamma s_{0} u_{1}) (1 - \gamma s_{0} u_{2}) + (1 - \gamma) (t - \gamma s_{0}^2) (1 - u_{1} u_{2}).$$
	\end{enumerate}
\end{lemma}

\begin{proof}
	
	To prove that $Z_{2n}$ is symmetric, let us introduce the vertex weights as in \Cref{fig:r_z_1}.
	\begin{figure}[htpb]
		\centering
		\includegraphics{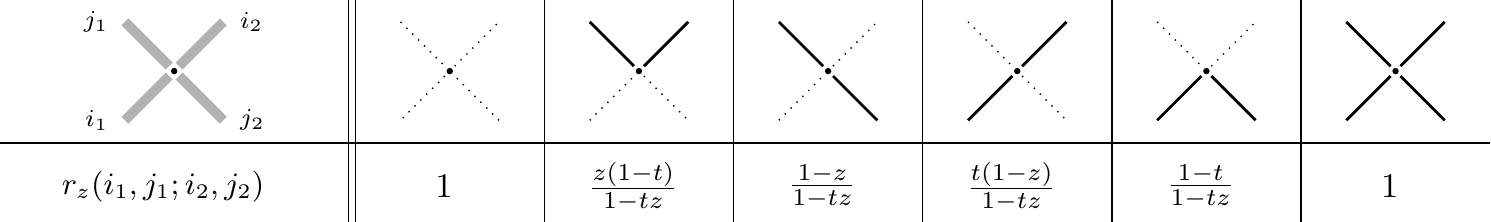}
		\label{fig:r_z}
	\end{figure}
	\begin{figure}[htpb]
		\centering
		\includegraphics{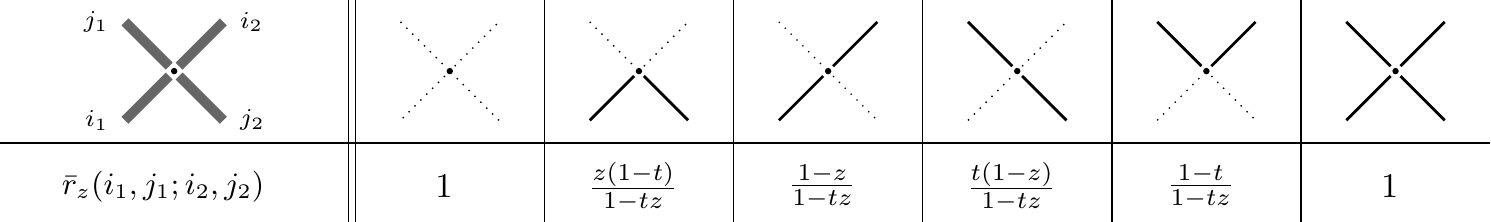}
		\caption{The vertex weights $r_z$ and $\bar{r}_z$ involved in the proof of symmetry. Here $i_1,j_1,i_2,j_2\in \left\{ 0,1 \right\}$.}
		\label{fig:r_z_1}
	\end{figure}
	One can check that they satisfy the following Yang--Baxter
	equations and the unitary relation:
	\begin{equation} \label{YB_r}
	\begin{split}
	&\sum_{k_1,k_2,k_3}r
	_{u/v}(i_2, i_1; k_2, k_1)\,
	R_{vw} (k_1, k_3; j_1, i_3)\,
	R_{uw}(k_2, j_3;j_2, k_3) \\
	&\hspace{100pt}
	= 
	\sum_{k'_1,k'_2,k'_3} R_{vw} (i_1, j_3; k'_1, k'_3)\,
	R_{uw}(i_2, k'_3;k'_2; i_3)
	r_{u/v}(k'_2, k'_1; j_2, j_1).
	\end{split}
	\end{equation}
	\begin{equation} \label{YB_r_bar}
	\begin{split}
	&\sum_{k_1,k_2,k_3}
	R_{uw} (i_3, i_2; k_3, k_2)\,
	R_{vw}(k_3, i_1;j_3, k_1)
	\bar{r}_{v/u}(k_2, k_1; j_2, j_1) \\
	&\hspace{100pt}
	= 
	\sum_{k'_1,k'_2,k'_3}
	\bar{r}
	_{v/u}(i_2, i_1; k'_2, k'_1)\,
	R_{uw} (k'_3, k'_2; j_3, j_2)\,
	R_{vw}(i_3, k'_1;k'_3, j_1).
	\end{split}
	\end{equation}
	\begin{equation}
	\label{eq:YBE_r_bar}
	\begin{split}
	&\sum_{k_1,k_2,k_3}\bar{r}_{v/u}(i_2, i_1; k_2, k_1)\,
	w^*_{v,s} (i_3, k_1; k_3, j_1)\,
	w^*_{u,s}(k_3,k_2; j_3,j_2) \\
	&\hspace{100pt}
	= 
	\sum_{k'_1,k'_2,k'_3} w^*_{v,s} (k'_3, i_1; j_3, k'_1)\,
	w^*_{u,s}(i_3,i_2; k'_3,k'_2)\,
	\bar{r}_{v/u}(k'_2, k'_1; j_2, j_1).
	\end{split}
	\end{equation}
	\begin{equation} \label{unit_r}
	\sum_{k_1,k_2,l_1, l_2}r
	_{u/v}(i_2, i_1; k_2, k_1)\,
	R_{uv} (k_1; k_2, l_1, l_2)\,
	\bar{r}_{v/u}(l_2; l_1,j_2, j_1) 
	= R_{uv} (i_2, i_1, j_2, j_1).
	\end{equation}
	
	Graphically, these equations can be viewed as in Figures \ref{fig:ybe_3_col} - \ref{fig:unit}:
	\begin{figure}[htpb]
		\centering
		\includegraphics{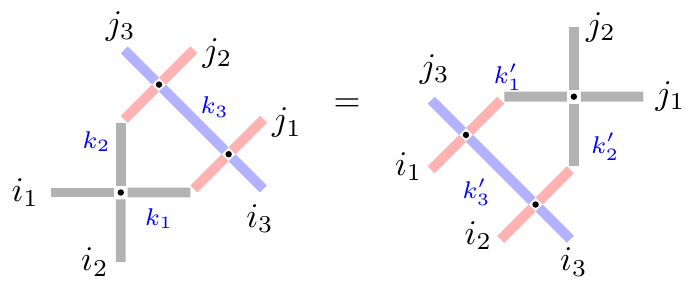}
		\caption{Graphical illustration of equation \eqref{YB_r}.}
		\label{fig:ybe_3_col}
	\end{figure}
	\begin{figure}[htpb]
		\centering
		\includegraphics{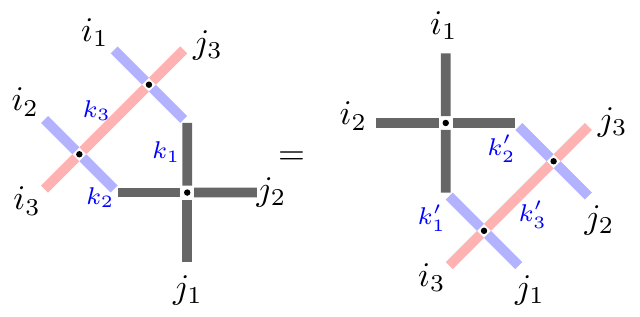}
		\caption{Graphical illustration of equation \eqref{YB_r_bar}.}
		\label{fig:ybe_3_col_d}
	\end{figure}
	\begin{figure}[htpb]
		\centering
		\includegraphics{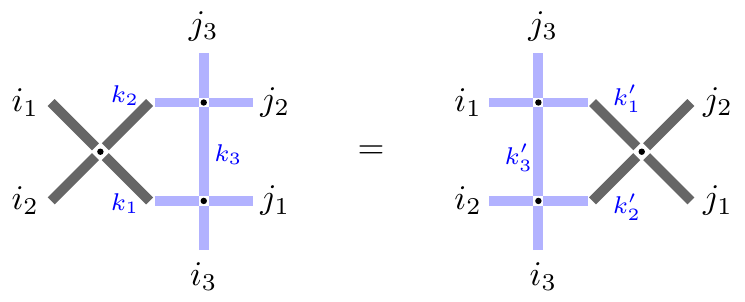}
		\caption{Graphical illustration of equation \eqref{eq:YBE_r_bar}.}
		\label{fig:ybe_r_bar}
	\end{figure}
	\begin{figure}[htpb]
		\centering
		\includegraphics{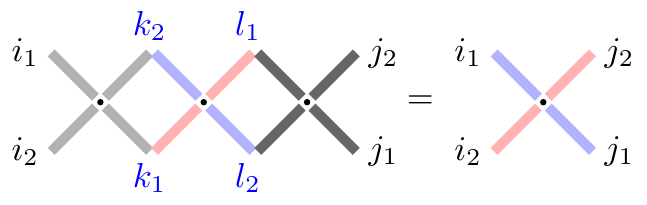}
		\caption{Graphical illustration of equation \eqref{unit_r}.}
		\label{fig:unit}
	\end{figure}
	
	To see property \textbf{1}, add to the partition function in \Cref{fig:5step} a vertex of weight $r_{u_{i+1}/ u_{i}}$ on the left at the $i$-th position, and a vertex of weight $\bar{r}_{u_{i}/u_{i+1}}$ on the right at the $i$-th position. On the one hand, these operations does not change the partition function at all. On the other hand, we can apply Yang--Baxter
	equations several times and the unitary relation to get the partition function with variables $u_i$ and  $u_{i+1}$ swapped. This concludes the proof of Property \textbf{1}.
	
	For property \textbf{2} we may assume that we are considering the same partition function but with the weights $(1 - u_i v_j)R_{u_iv_j}$ instead of $R_{u_iv_j}$ and $(1 - s_0 u_j) w^{*}_{u_j}$ instead of $ w^{*}_{u_j}$. This makes all the weights linear, in particular, in $u_{2n}$. It allows to verify that $u_{2n}$ contributes to each part of the summation $2n-1$ times with some coefficients (independent of $u_{2n}$).
	
	For property \textbf{3} let us notice that in this case $(1-u_{2n}u_{2n-1})R_{u_{2n}u_{2n-1}}(1,1;1,1) = 0$, so we should avoid this weight at the beginning, which leads to factorization of the partition function. After some computations we get the desired property.
	
	To prove property \textbf{4}, note that the chosen $u_1, \dots, u_{2n}$ satisfy 
	$
	1 - t u_i u_j = 0
	$ for all $i, j$ such that $i+j$  is odd. So, for odd $i+j$ we have $R_{u_i v_j}(0,1,0,1) = R_{u_i v_j}(1,0,1,0) = 0$ and $R_{u_i v_j}(k_1,k_1,k_2,k_2) = (-1)^{k_1 + k_2} t^{k_1}$.
	This observation together with some simple freezing/combinatorial arguments implies that there are $2^{n}$ possible configurations with non-zero weights. Moreover, they are uniquely determined by values on the right edge of the six vertex model. One can compute explicitly the weight of each configuration. For example, it can be done through the following recursion relation:
	\begin{equation*}
	\begin{split}
	&
	Z_{2n}(t,  1 / t^2, \dots, t,   1 / t^2) \\&
	= 
	(1 - t) (t^{2n-1} (1 - s_{0}^2 \gamma/t)(1 - \gamma) + (-t)^{2n-1}(1 - s_0 \gamma/ t^{2})(1 - s_0 \gamma t)) 
	\\&\hspace{160pt}\times \prod_{i<2n-1} (1 - u_i u_{2n}) (1 - u_i u_{2n-1})
	Z_{2n-2} (t,  1 / t^2, \dots, t,   1 / t^2).
	\end{split}
	\end{equation*}
	Here the first and the second summands correspond to the cases where we choose $R_{u_{2n} u_{1}}(1, 1, 1, 1)$ or $R_{u_{2n} u_{1}}(1, 1, 0, 0)$ cross vertex weights, respectively. 
	This recursion immediately gives us the formula:
	\begin{equation*}
	Z_{2n}(t,  1 / t^2, \dots, t,   1 / t^2)
	=  \gamma^{n} t^{n(n-2)} (1 - s_{0} t^{-2})^n (1 -s_{0} t)^{n}\prod_{i<j} (1 - u_i u_j)|_{(u_1, ... , u_{2n}) = (t,  1 / t^2, \dots, t,   1 / t^2)}.
	\end{equation*}
	Finally, property {\bf 5} comes from direct computations.
\end{proof}

\subsection{Explicit formula for the partition function}

\begin{theorem} \label{pf_thm}
	The partition function $\mathcal{P} (u_1, \dots, u_{2n}; t, \alpha)$ can be expressed explicitly as follows:
	\begin{equation}
	\label{pf_p}
	\mathcal{P} (u_1, \dots, u_{2n}; t, \alpha)
	= 	\prod_{j = 1}^{2n} \frac{1}{1 - s_0 u_j} \prod_{1 \leq i<j \leq 2n}
	\left(
	\frac{1-t u_i u_j}{u_i - u_j}
	\right)
	\pf_{1\leq i < j \leq 2n}
	\left[
	\frac{Z_{2}(u_i, u_j) (u_i - u_j)}
	{(1-u_i u_j) (1-t u_i u_j)}
	\right].
	\end{equation}
\end{theorem}
\begin{proof}
	First, one can deduce that the Pfaffian on the right-hand side of \eqref{pf_p} multiplied by the product $\prod_{j = 1}^{2n} (1 - s_0 u_j) \prod_{1 \leq i < j \leq 2n} (1 - u_{i} u_{j}) $ satisfy all the properties \textbf{1}-\textbf{5} from Lemma~\ref{prop_lem}. Namely, we have property \textbf{1} because both the Pfaffian and the Vandermonde $\prod_{1 \leq i < j \leq 2n} (u_{i} - u_{j}) $ change the sign under swaps $u_i \leftrightarrow u_{i+1}$. Properties \textbf{2} and \textbf{5} are straightforward from the very definition of the Pfaffian. To get property \textbf{3}, one can multiply the last row and column by $\prod_{1 \leq j < 2n} (1 - u_{j} u_{2n}) $ and the second-to-last row and column by $(1 - u_{2n-1} u_{2n})\prod_{1 \leq j < 2n-1} (1 - u_{j} u_{2n-1}) $. Note that all the elements in this matrix hook vanish except two with indices $n-1$ and $n$. In turn,
	\begin{equation*}
	Z_{2}(u_{2n-1},u_{2n})\big\vert_{u_{2n}=u_{2n-1}^{-1}}= (1-t)
	(1-\gamma s_0u_{2n})(1-\gamma s_0 u_{2n-1}).
	\end{equation*}
	Likewise, one can prove property \textbf{4}, using the recurrence and the following: \begin{equation*}Z_{2}(t,  1 / t^2) = \gamma(1 - t)(1 - s_{0} t^{-2})(1 -s_{0} t).\end{equation*}
	So, it remains to show that these properties determine a function uniquely. For this purpose one can use Lagrange interpolation the same way as in  \cite{Pet21} and \cite[Appendix B]{WZJ16}.
	
	Namely, we assume that two families of polynomials $f_{2n}(u_1, \dots, u_{2n})$ and $g_{2n}(u_1, \dots, u_{2n})$ satisfy properties \textbf{1}-\textbf{5} and prove that $f_{2n} = g_{2n}$ by induction on $n$. The base case  follows from Property \textbf{5}. To prove the induction step, assume that we proved this statement for $n-1$. Then, let us fix $2n-1$ arbitrary non-zero distinct points $u_1, \dots, u_{2n-1}$. Using the recurrence relation and symmetry, we obtain that $f_{2n}$ and $g_{2n}$ treated as polynomials in $u_{2n}$ coincide in $2n-1$  distinct points $u_{1}^{-1}, \dots, u_{2n-1}^{-1}$. Since their degree is $2n-1$, it follows that $f_{2n} - g_{2n} = c \cdot \prod_{i<j}(1 - u_i u_j)$ where $c$ does not depend on $u_{2n}$. However, because of symmetry it does not depend on $u_1, \dots, u_{2n-1}$ either, which means it is an absolute constant. Finally, as can be seen from property \textbf{4}, $f_{2n}$ and $g_{2n}$ have the same value at a fixed point, hence $f_{2n} = g_{2n}$.
	
	This concludes the proof of Theorem~\ref{pf_thm}.
\end{proof}

\begin{corollary}
	Under the specialization $u_j = t^{2n - j} / (\gamma s_{0})$, we have 
	\begin{equation} \label{pf_p_sp}
	\begin{split}
	& 
	\pf_{1\leq i < j \leq 2n}
	\left[
	\frac{ (t^{j} - t^{i})(\gamma^2 s_0^2(1 - t)(t^{i} - t^{2n})(t^{j}- t^{2n}) + (1 - \gamma)(t-s_{0}^{2}\gamma)(t^{i+j}\gamma^2 s_0^2-t^{4n})}
	{t^{2i+2j-2n}\gamma s_0(\gamma s_0-t^{4n}/(\gamma s_0)) (\gamma s_0 -t^{4n+1-i-j}/(\gamma s_0))}
	\right]
	\\
	&\hspace{-3pt}
	= \prod_{j = 0}^{2n-1} (1 - t^j\gamma^{-1}) \prod_{0 \leq i<j \leq 2n-1}
	\left(
	\frac{t^j - t^i}{\gamma s_0 -t^{i+j+1} / (\gamma s_0)}
	\right)
	\mathcal{P}(t^{2n - 1} / (\gamma s_{0}), t^{2n - 2} / (\gamma s_{0}), \dots, 1 / (\gamma s_{0}); t, \alpha)  \\
	&\hspace{65pt}=
	(-1)^{n} \gamma^{n} t^{n^{2}} \prod_{0 \leq i<j \leq 2n-1}
	\left(
	\frac{t^j - t^i}{\gamma s_0 -t^{i+j+1} / (\gamma s_0)}
	\right)
	\prod_{j=1}^{n}(1 - s_{0}^{2}\gamma t^{-2j+1})(1 - \gamma^{-1} t^{2j-2})
	.
	\end{split}
	\end{equation}
\end{corollary}
\begin{proof}
	Consider the lattice interpretation of $\mathcal{P}(t^{2n - 1} / (\gamma s_{0}), \dots, 1 / (\gamma s_{0}); t, \alpha)$. Indeed, under this specialization the right boundary is fixed, and therefore the whole configuration becomes frozen.
	By the way, it is not so easy to verify independently that the Pfaffian in the left-hand side of \eqref{pf_p_sp} factorizes.
\end{proof}

Using \eqref{part_alpha} and \eqref{F_alpha}, the left-hand side of the identity \eqref{pf_p} can be rewritten as the weighted sum of $F_{\lambda}$'s over all signatures $\lambda$ with even multiplicities in the following way:
\begin{equation*}
\begin{split}
&
\sum_{\lambda: m_{i}(\lambda) \in 2 \mathbb{Z}_{\geq 0}} \frac{1}{( t; t)_{m_{0}(\lambda)}} \prod_{j = 1}^{m_{0}(\lambda) / 2} (1 - s_{0}^2 \gamma t^{2j - 2})(1 - \gamma t^{2j-1}) \prod_{j=1}^{2n} \frac{1 - \gamma s_0 u_j}{1 - s_0 u_j}
\\
&\hspace{200pt}
\times \prod_{i = 1}^{\infty} \prod_{j = 1}^{m_{i}(\lambda) / 2} \frac{1 - s_{i}^2 t^{2j - 2}}{1 - t^{2j}} \left[F_{\lambda} (u_1, \dots, u_{2n})\Big\vert_{\text{$s_0\to \gamma s_0$}}\right].
\end{split}
\end{equation*}

After replacing $\gamma s_0$ by $s_0$, we obtain the desired statement of Theorem~\ref{lit_id}.

\section{Some specializations of the Littlewood identity} \label{sec: spec}

In this section we reduce our result
to classical Hall--Littlewood polynomials and we write a non-refined degeneration of our result.
\subsection{Reduction to the case of classical Hall--Littlewood polynomials}
As was shown in \cite{Pet21}, spin Hall--Littlewood rational functions $F_{\lambda}$ can be reduced to classical Hall--Littlewood polynomials $P_\lambda^{HL}$ in the following way:
\begin{equation}
F_{\lambda}(u_1, \dots, u_N)\Big\vert_{\text{$s_x=0$}} = \prod_{r\ge0}(t;t)_{m_r(\lambda)}\cdot
P_\lambda^{HL}(u_1,\ldots,u_N).
\end{equation}
So, after setting $ s_x = 0$ in equation~\eqref{littlewood_id}, we get
\begin{multline} \label{class_id}
\sum_{\lambda: m_{i}(\lambda) \in 2 \mathbb{Z}_{\geq 0}}  \prod_{j = 1}^{m_{0}(\lambda) / 2}(1 - \gamma t^{2j-1})\prod_{i = 1}^{\infty} \prod_{j = 1}^{m_{i}(\lambda) / 2} (1 - t^{2j-1}) P_{\lambda}^{HL} (u_1,\ldots,u_{2n})
=\\
=
\prod_{1 \leq i<j \leq 2n}
\left(
\frac{1-t u_i u_j}{u_i - u_j}
\right) 
\pf_{1\leq i < j \leq 2n}
\left[
\frac{ (u_i - u_j)((1 - \gamma t+ (\gamma-1)tu_i u_j)}
{(1-u_i u_j) (1-t u_i u_j)}
\right],
\end{multline}
which coincides with the Littlewood identity proved in \cite[Theorem 5]{WZJ16}.

\subsection{Reduction to the unrefined case}
To get unrefined identity, we set $\gamma = 1$ and obtain the following formula:
\begin{multline} \label{unref_id}
\sum_{\lambda: m_{i}(\lambda) \in 2 \mathbb{Z}_{\geq 0}}   \prod_{i = 0}^{\infty} \prod_{j = 1}^{m_{i}(\lambda) / 2} \frac{1 - s_{i}^2 t^{2j - 2}}{1 - t^{2j}} F_{\lambda} (u_1, \dots, u_{2n})
=\\
=
\prod_{1 \leq i<j \leq 2n}
\left(
\frac{1-t u_i u_j}{u_i - u_j}
\right) 
\pf_{1\leq i < j \leq 2n}
\left[
\frac{ (u_i - u_j)(1 - t) }
{(1-u_i u_j) (1-t u_i u_j)}
\right].
\end{multline}

The right-hand side of \eqref{unref_id} coincides with the right-hand side of \eqref{class_id} at $\gamma = 1$, but the expansions are different.

\bigskip

\nocite{*}
\bibliographystyle{alpha}
\bibliography{fin_2_ref_littlewood}

\begin{thebibliography}{BBCW18}

\bibitem[BBC20]{BBC20}
Guillaume {Barraquand}, Alexei {Borodin}, and Ivan {Corwin}.
\newblock {Half-space Macdonald processes}.
\newblock {\em {Forum Math. Pi}}, 8:150, 2020.
\newblock Id/No e11.

\bibitem[BBCW18]{BBCW18}
Guillaume {Barraquand}, Alexei {Borodin}, Ivan {Corwin}, and Michael {Wheeler}.
\newblock {Stochastic six-vertex model in a half-quadrant and half-line open
  asymmetric simple exclusion process}.
\newblock {\em {Duke Math. J.}}, 167(13):2457--2529, 2018.

\bibitem[BCPS15]{BCPS15}
Alexei {Borodin}, Ivan {Corwin}, Leonid {Petrov}, and Tomohiro {Sasamoto}.
\newblock {Spectral theory for interacting particle systems solvable by
  coordinate Bethe ansatz}.
\newblock {\em {Commun. Math. Phys.}}, 339(3):1167--1245, 2015.

\bibitem[BCPS19]{BCPS15c}
Alexei {Borodin}, Ivan {Corwin}, Leonid {Petrov}, and Tomohiro {Sasamoto}.
\newblock {Correction to: ``Spectral theory for interacting particle systems
  solvable by coordinate Bethe ansatz''}.
\newblock {\em {Commun. Math. Phys.}}, 370(3):1069--1072, 2019.

\bibitem[{Bor}17]{Bor17}
Alexei {Borodin}.
\newblock {On a family of symmetric rational functions}.
\newblock {\em {Adv. Math.}}, 306:973--1018, 2017.

\bibitem[{Bor}18]{Bor18}
Alexei {Borodin}.
\newblock {Stochastic higher spin six vertex model and Macdonald measures}.
\newblock {\em {J. Math. Phys.}}, 59(2):023301, 17, 2018.

\bibitem[BP18]{BP18}
Alexei {Borodin} and Leonid {Petrov}.
\newblock {Higher spin six vertex model and symmetric rational functions}.
\newblock {\em {Sel. Math., New Ser.}}, 24(2):751--874, 2018.

\bibitem[BZ19]{BZ19}
Elia {Bisi} and Nikos {Zygouras}.
\newblock {Point-to-line polymers and orthogonal Whittaker functions}.
\newblock {\em {Trans. Am. Math. Soc.}}, 371(12):8339--8379, 2019.

\bibitem[CD21]{CD21}
Kailun {Chen} and Xiangmao {Ding}.
\newblock Stable spin hall-littlewood symmetric functions, combinatorial
  identities, and half-space yang-baxter random field, 2021.
\newblock \textit{arXiv:2106.12557}.

\bibitem[CP16]{CP16}
Ivan {Corwin} and Leonid {Petrov}.
\newblock {Stochastic higher spin vertex models on the line}.
\newblock {\em {Commun. Math. Phys.}}, 343(2):651--700, 2016.

\bibitem[CP19]{CP16c}
Ivan {Corwin} and Leonid {Petrov}.
\newblock {Correction to: ``Stochastic higher spin vertex models on the
  line''}.
\newblock {\em {Commun. Math. Phys.}}, 371(1):353--355, 2019.

\bibitem[{Ize}87]{Ize87}
A.~G. {Izergin}.
\newblock {Partition function of a six-vertex model in a finite volume}.
\newblock {\em {Dokl. Akad. Nauk SSSR}}, 297(2):331--333, 1987.

\bibitem[KBI93]{KBI93}
V.~E. {Korepin}, N.~M. {Bogoliubov}, and A.~G. {Izergin}.
\newblock {\em {Quantum inverse scattering method and correlation functions}}.
\newblock Cambridge: Cambridge University Press, 1993.

\bibitem[KN99]{KN99}
A.~N. {Kirillov} and M.~{Noumi}.
\newblock {\(q\)-difference raising operators for Macdonald polynomials and the
  integrality of transition coefficients}.
\newblock In {\em {Algebraic methods and \(q\)-special functions}}, pages
  227--243. Providence,~RI: American Mathematical Society, 1999.

\bibitem[{Kup}02]{Kup02}
Greg {Kupferberg}.
\newblock {Symmetry classes of alternating-sign matrices under one roof}.
\newblock {\em {Ann. Math. (2)}}, 156(3):835--866, 2002.

\bibitem[{Mac}95]{Mac95}
Ian~Grant {Macdonald}.
\newblock {\em {Symmetric functions and Hall polynomials. 2nd ed}}.
\newblock Oxford: Clarendon Press, 2nd ed. edition, 1995.

\bibitem[{Pet}21]{Pet21}
Leonid {Petrov}.
\newblock {Refined Cauchy identity for spin Hall-Littlewood symmetric rational
  functions}.
\newblock {\em {J. Comb. Theory, Ser. A}}, 184:50, 2021.
\newblock Id/No 105519.

\bibitem[{Pov}13]{Pov13}
A.~M. {Povolotsky}.
\newblock {On the integrability of zero-range chipping models with factorized
  steady states}.
\newblock {\em {J. Phys. A, Math. Theor.}}, 46(46):25, 2013.
\newblock Id/No 465205.

\bibitem[RW21]{RW21}
Eric~M. {Rains} and S.~Ole {Warnaar}.
\newblock {\em {Bounded Littlewood identities}}.
\newblock Providence, RI: American Mathematical Society (AMS), 2021.

\bibitem[{War}06]{War06}
S.~Ole {Warnaar}.
\newblock {Rogers-Szeg\H{o} polynomials and Hall-Littlewood symmetric
  functions}.
\newblock {\em {J. Algebra}}, 303(2):810--830, 2006.

\bibitem[{War}08]{War08}
S.~Ole {Warnaar}.
\newblock {Bisymmetric functions, Macdonald polynomials and $\mathfrak{sl}_3$
  basic hypergeometric series}.
\newblock {\em {Compos. Math.}}, 144(2):271--303, 2008.

\bibitem[WZ16]{WZJ16}
Michael {Wheeler} and Paul {Zinn-Justin}.
\newblock {Refined Cauchy/Littlewood identities and six-vertex model partition
  functions. III. Deformed bosons}.
\newblock {\em {Adv. Math.}}, 299:543--600, 2016.

\end{thebibliography}
\medskip

\textsc{S. Gavrilova, Faculty of Mathematics, National Research University Higher School of Economics, Usacheva 6, 119048 Moscow, Russia}

E-mail: \texttt{sveta\_6117@mail.ru}

\end{document}